\documentclass[11pt]{article}
\usepackage{amssymb,latexsym,amsmath,amsbsy,amsthm,amsxtra,amsgen,graphicx}
\newfont{\msbm}{msbm10 at 11pt}
\newcommand {\R} {\mbox{\msbm R}}

\newtheorem{Theo}{Theorem}
\newtheorem{Lemma}[Theo]{Lemma}
\newtheorem{Cor}[Theo]{Corollary}
\newtheorem{Prop}[Theo]{Proposition}

\DeclareMathOperator{\E}{\mathbb{E}}
\renewcommand{\P}{}
\let\P\relax
\DeclareMathOperator{\P}{\mathbb{P}}

\begin{document}
\title{Critical branching Brownian motion with absorption: \\
survival probability}

\author{Julien Berestycki\thanks{Research supported in part by ANR-08-BLAN-0220-01 and ANR-08-BLAN-0190 and ANR-08-BLAN-0190}, Nathana\"el Berestycki\thanks{Research supported in part by EPSRC grants EP/GO55068/1 and EP/I03372X/1}
  \ and Jason Schweinsberg\thanks{Supported in part by NSF Grants DMS-0805472 and DMS-1206195}}
\maketitle

\footnote{{\it MSC 2010}.  Primary 60J65; Secondary 60J80, 60J25}

\footnote{{\it Key words and phrases}.  Branching Brownian motion, extinction time, survival probability, critical phenomena}

\vspace{-.6in}
\begin{abstract}
We consider branching Brownian motion on the real line with absorption at zero, in which particles move according to independent Brownian motions with the critical drift of $-\sqrt{2}$.  Kesten (1978) showed that almost surely this process eventually dies out.  Here we obtain upper and lower bounds on the probability that the process survives until some large time $t$.  These bounds improve upon results of Kesten (1978), and partially confirm nonrigorous predictions of Derrida and Simon (2007).
\end{abstract}

\section{Introduction}

\subsection{Main results}

We consider branching Brownian motion with absorption, which is constructed as follows.  At time zero, there is a single particle at $x > 0$.  Each particle moves independently according to one-dimensional Brownian motion with a drift of $-\mu$, and each particle independently splits into two at rate $1$.  Particles are killed when they reach the origin.  This process was first studied in 1978 by Kesten \cite{kesten}, who showed that almost surely all particles are eventually killed if $\mu \geq \sqrt{2}$, whereas with positive probability there are particles alive at all times if $\mu < \sqrt{2}$.  Thus, $\mu = \sqrt{2}$ is the critical value for the drift parameter.

Harris, Harris, and Kyprianou \cite{hhk06} obtained an asymptotic result for the survival probability of this process when $\mu < \sqrt{2}$.  Harris and Harris \cite{hh07} focused on the subcritical case $\mu > \sqrt{2}$ and estimated the probability that the process survives until time $t$ for large values of $t$.  Results about the survival probability in the nearly critical case when $\mu$ is just slightly larger than $\sqrt{2}$ were obtained in \cite{bbm2, ds07, sd08}.  Questions about the survival probability have likewise been studied for branching random walks in which particles are killed when they get below a barrier.  See \cite{aija11, bego09, fz, ghs, jaffuel} for recent progress in this area.

In this paper, we consider the critical case in which $\mu = \sqrt{2}$.  Let $\zeta$ be the time when the process becomes extinct, which we know is almost surely finite.
Kesten showed (see Theorem 1.3 of \cite{kesten}) that there exists $K > 0$ such that for all $x > 0$, we have $$x e^{\sqrt{2}x - K (\log t)^2 - (3 \pi^2 t)^{1/3}} \leq \P(\zeta > t) \leq (1 + x) e^{\sqrt{2}x + K(\log t)^2 - (3 \pi^2 t)^{1/3}}$$ for sufficiently large $t$.  Our main result, which is Theorem \ref{yaglom} below, improves upon this result.
For this result, and throughout the rest of the paper, we let
\begin{equation}\label{cdef}
\tau = \frac{2 \sqrt{2}}{3 \pi^2}, \hspace{.5in} c = \tau^{-1/3} = \bigg( \frac{3 \pi^2}{2 \sqrt{2}} \bigg)^{1/3}.
\end{equation}

\begin{Theo}\label{yaglom}
There exist positive constants $C_1$ and $C_2$ such that
\begin{equation}\label{E:ext proba 1}
C_1 e^{\sqrt{2} x} \sin\bigg( {\pi x\over ct^{1/3}}\bigg) t^{1/3}e^{-(3 \pi^2 t)^{1/3}} \leq \P(\zeta>t) \leq C_2 e^{\sqrt{2} x} \sin\bigg( {\pi x\over ct^{1/3}}\bigg) t^{1/3}e^{-(3 \pi^2 t)^{1/3}}
\end{equation}
for any $x > 0$ and $t > 0$ such that $x<ct^{1/3}-1.$ 
In particular, there exist positive constants $C_3$ and $C_4$ such that for any fixed $x>0$, we have
\begin{equation}\label{E:ext proba 2}
C_3 x e^{\sqrt{2} x} e^{-(3 \pi^2 t)^{1/3}} \leq \P(\zeta > t) \leq C_4 x e^{\sqrt{2} x} e^{-(3 \pi^2 t)^{1/3}}
\end{equation}
for sufficiently large $t$.
\end{Theo}

The main novelty in Theorem \ref{yaglom} is that the terms $e^{\pm K (\log t)^2}$ in Kesten's upper and lower bounds may be replaced by constants $C_1$ and $C_2$ respectively. Nonrigorous work of Derrida and Simon \cite{ds07} indicates that it should be possible to obtain a result even sharper than Theorem \ref{yaglom}.  Indeed, equation (13) of \cite{ds07} indicates that for each fixed $x$, we should have $$\P(\zeta > t) \sim C e^{-(3 \pi^2 t)^{1/3}}$$ as $t \rightarrow \infty$, where $C$ is a constant depending on $x$.

Note that the result (\ref{E:ext proba 1}) is only valid when $0 < x < ct^{1/3} - 1$.  However, when $x=ct^{1/3}-1$, equation \eqref{E:ext proba 1} shows that the survival probability up to $t$ is already of order 1.
It is an open question whether there exists a function $\phi : \R \mapsto [0,1]$ such that
$$
\P_{ct^{1/3}+x} (\zeta >t) \to \phi(x)
$$
as $t\to \infty$,
where $\P_z$ denotes probabilities for branching Brownian motion started from a single particle at $z$.

An important tool in the proof of Theorem \ref{yaglom} will be the following result of independent interest, which gives sharp estimates on the extinction time of the process when the position $x$ of the initial particle tends to infinity.

\begin{Theo}\label{extinct}
Let $\varepsilon > 0$.  Then there exists a positive number $\beta > 0$, depending on $\varepsilon$, such that for sufficiently large $x$, $$\P \big(\tau x^3 - \beta x^2 < \zeta < \tau x^3 + \beta x^2 \big) \geq 1 - \varepsilon.$$
\end{Theo}

Let $x>0$ and let $t = \tau x^3$. Thus Theorem \ref{extinct} says that if there is initially one particle at $x$, the extinction time of the process will be close to $t$ (if $x$ is large). Conversely, fix $t>0$ and define a function 
\begin{equation}
\label{Ldef}L(s) = c (t-s)^{1/3}.
\end{equation}
From Theorem \ref{extinct}, we see that if a particle reaches $L(s)$ at time $s \in (0, t)$, then there is a good chance that a descendant of this particle will survive until time $t$.  Our strategy for proving Theorem \ref{yaglom} will be to estimate the probability that a particle reaches $L(s)$ for some $s \in (0,t)$, and then argue that, up to a constant, this is the same as the probability that the process survives until time $t$.

Theorem \ref{yaglom} gives an estimate of the probability that the process started with one particle at $x > 0$ survives until some large time $t$.  An important open question is to determine, conditional on survival up to a large time $t$, what the configuration of particles will look like before time $t$.  The complete description of the configuration of particles, conditionally upon survival up to a large time $t$, is known as the Yaglom conditional limit. This is in turn related to a main conjecture concerning the limiting behaviour of the Fleming-Viot process proposed by Burdzy et al. \cite{burdzy1, burdzy2}. See \cite{AFGJ} for a recent discussion and verification in a particular case of that conjecture.

This is the first in a series of two papers concerning the properties of critical branching Brownian motion with absorption. In the companion paper \cite{bbs3}, we use ideas developed in this paper to obtain a precise description of the particle configuration at times $0 \le s \le t$, when the position $x$ of the initial particle tends to infinity and $t = \tau x^3$.  It seems likely that the results and methods of \cite{bbs3} will also shed some light on the behavior of the process conditioned to survive for a long time.


\subsection{Organization of the paper} 
In Sections \ref{stripsec} and \ref{stripsec2}, we collect some general results about branching Brownian motion killed at the boundaries of a strip.  Theorem \ref{yaglom} and Theorem \ref{extinct} are proved in Section \ref{yagsec}.  Throughout the paper, $C$ will denote a positive constant whose value may change from line to line, and $\asymp$ will mean that the ratio of the two sides is bounded above and below by positive constants.

\section{Branching Brownian motion in a strip}\label{stripsec}

We collect in this section some results pertaining to branching Brownian motion in a strip.  Consider branching Brownian motion in which each particle drifts to the left at rate $-\sqrt{2}$, and each particle independently splits into two at rate $1$.  Particles are killed if either they reach $0$ or if they reach $L(s)$ at time $s$, where $L(s) \geq 0$ for all $s$.  We assume that the initial configuration of particles is deterministic, with all particles located between $0$ and $L(0)$.

Let $N(s)$ be the number of particles at time $s$, and denote the positions of the particles at time $s$ by $X_1(s) \geq X_2(s) \geq \dots \geq X_{N(s)}(s)$.
Let
$$Z(s) = \sum_{i=1}^{N(s)} e^{\sqrt{2} X_i(s)} \sin \bigg( \frac{\pi X_i(s)}{L(s)} \bigg).$$
Let $({\cal F}_s, s \geq 0)$ denote the natural filtration associated with the branching Brownian motion.

Let $q_s(x,y)$ denote the density of the branching Brownian motion, meaning that if initially there is a single particle at $x$ and $A$ is a Borel subset of $(0,L(s))$, then the expected number of particles in $A$ at time $s$ is $$\int_A q_s(x,y) \: dy.$$

\subsection{A constant right boundary}

We first consider briefly the case in which $L(s) = L$ for all $s$, which was studied in \cite{bbs}.  The following result is Lemma 5 of \cite{bbs}.

\begin{Lemma}\label{stripdensity}
For $s > 0$ and $x,y \in (0, L)$, let $$p_s(x,y) = \frac{2}{L} e^{-\pi^2 s/2L^2} e^{\sqrt{2} x} \sin \bigg( \frac{\pi x}{L} \bigg) e^{-\sqrt{2} y} \sin \bigg( \frac{\pi y}{L} \bigg),$$ and define $D_s(x,y)$ so that $$q_s(x,y) = p_s(x,y)(1 + D_s(x,y)).$$  Then for all $x,y \in (0,L)$, we have $$|D_s(x,y)| \leq \sum_{n=2}^{\infty} n^2 e^{-\pi^2(n^2 - 1)s/2L^2}.$$
\end{Lemma}

Lemma \ref{stripdensity} allows us to approximate $q_s(x,y)$ by $p_s(x,y)$ when $s$ is sufficiently large.  We will also use the following result, which follows from (28) and (51) of \cite{bbs} and is proved using Green's function estimates for Brownian motion in a strip.

\begin{Lemma} \label{mainqslem}
For all $s \geq 0$ and all $x, y \in (0, L)$, we have
$$\int_0^{\infty} q_s(x,y) \: ds \leq \frac{2 e^{\sqrt{2}(x-y)} x (L - y)}{L}.$$
\end{Lemma}

\subsection{A piecewise linear right boundary}

Fix $m > 0$, and fix $0 < K < L$.  Also, let $t > 0$.  We consider here the case in which
\begin{displaymath}
L(s) = \left\{
\begin{array}{ll}  L & \mbox{ if }0 \leq s \leq t - m^{-1}(L-K)  \\
K + m(t-s) & \mbox{ if } t - m^{-1}(L-K) \leq s \leq t.
\end{array} \right.
\end{displaymath}
We will assume that $m^{-1}(L-K) \leq t/2$.  Thus, the right boundary stays at $L$ from time $0$ until at least time $t/2$, but eventually moves to the left at a linear rate, reaching $K$ at time $t$.

To obtain an estimate of $q_s(x,y)$, we will need the following result for the probability that a Brownian bridge crosses a line.  This result is well-known and follows immediately, for example, from Proposition 3 of \cite{scheike}.  We let  $B^{br}_{x,y,t} = (B^{br}_{x,y,t}(s), 0 \leq s \leq t)$ denote the Brownian bridge from $x$ to $y$ of length $t$.
\begin{Lemma}
If $x < a$ and $y < a + bt$, then
\begin{equation}\label{bbr1}
\P(B^{br}_{x,y,t}(s) \geq a + bs \mbox{ for some }s \in [0,t]) = \exp\bigg(-\frac{2(a-x)(a + bt - y)}{t}\bigg).
\end{equation}
If $x > a$ and $y > a + bt$, then
\begin{equation}\label{bbr2}
\P(B^{br}_{x,y,t}(s) \leq a + bs \mbox{ for some }s \in [0,t]) = \exp\bigg(-\frac{2(x-a)(y - a - bt)}{t}\bigg).
\end{equation}
\end{Lemma}

\begin{proof}
Proposition 3 of \cite{scheike} states that if $a > 0$ and $y < a + bt$, then
$$\P(B^{br}_{0,y,t}(s) \geq a + bs \mbox{ for some }s \in [0,t]) = \exp\bigg(-\frac{2a(a + bt - y)}{t}\bigg).$$  The result (\ref{bbr1}) follows because $(B^{br}_{0,y,t}(s) + x(t-s)/t, 0 \leq s \leq t)$ is a Brownian bridge of length $t$ from $x$ to $y$.  Then (\ref{bbr2}) follows because $(-B^{br}_{x,y,t}(s), 0 \leq s \leq t)$ is a Brownian bridge of length $t$ from $-x$ to $-y$.
\end{proof}

\begin{Lemma}\label{pllem}
There exists a positive constant $C$ such that if $t > 0$ and $K + mt/2 \leq 2L$, then for all $x \in [0,L]$ and all $y \in [0,K]$, we have $$q_t(x,y) \leq \frac{CL^4}{t^{5/2}} e^{\sqrt{2} x} \sin \bigg( \frac{\pi x}{L} \bigg) e^{-\sqrt{2} y} \sin \bigg( \frac{\pi y}{K} \bigg).$$
\end{Lemma}

\begin{proof}
First, we claim that $$q_t(x,y) = \frac{1}{\sqrt{2 \pi t}} e^{-(x - y)^2/2t} \cdot e^{\sqrt{2}(x-y) - t} \cdot e^t \cdot \P(0 \leq B^{br}_{x,y,t}(s) \leq L(s) \mbox{ for all }s \in [0,t]).$$  To see this, observe that the first factor is the density for standard Brownian motion, the second factor is a Girsanov term that relates Brownian motion with drift to standard Brownian motion, the third factor of $e^t$ accounts for the branching at rate 1, and the fourth factor is the probability that a Brownian particle that starts at $x$ and ends at $y$ avoids being killed at one of the boundaries.  Therefore,
\begin{equation}\label{qtbrint}
q_t(x,y) \leq \frac{C e^{\sqrt{2}(x-y)}}{\sqrt{t}} \P(0 \leq B^{br}_{x,y,t}(s) \leq L(s) \mbox{ for all }s \in [0,t]).
\end{equation}
Let $g$ denote the density of $B^{br}_{x,y,t}(t/2)$.  Then
\begin{align}\label{brintg}
&\P(0 \leq B^{br}_{x,y,t}(s) \leq L(s) \mbox{ for all }s \in [0,t]) \nonumber \\
&\qquad= \int_0^{L(t/2)} \P(0 \leq B^{br}_{x,z,t/2}(s) \leq L(s) \mbox{ for all }s \in [0, t/2]) \nonumber \\
&\qquad \qquad \times \P(0 \leq B^{br}_{z,y,t/2} \leq L(t/2 + s) \mbox{ for all }s \in [0, t/2]) g(z) \: dz.
\end{align}
Recall that $L(s) = L$ for all $s \in [0, t/2]$.  Therefore, if $0 \leq x \leq L/2$ and $0 \leq z \leq L$, then by (\ref{bbr2}) with $a = b = 0$,
\begin{align}\label{br1}
\P(0 \leq B^{br}_{x,z,t/2}(s) \leq L(s) \mbox{ for all }s \in [0, t/2]) &\leq \P(B^{br}_{x,z,t/2}(s) \geq 0 \mbox{ for all }s \in [0, t/2]) \nonumber \\
&= 1 - \P(B^{br}_{x,z,t/2}(s) \leq 0 \mbox{ for some }s \in [0, t/2]) \nonumber \\
&= 1 - \exp \bigg(- \frac{4xz}{t} \bigg) \nonumber \\
&\leq \frac{4xL}{t}.
\end{align}
If $L/2 \leq x \leq L$ and $0 \leq z \leq L$, then by (\ref{bbr1}) with $a = L$ and $b = 0$,
\begin{align}\label{br2}
\P(0 \leq B^{br}_{x,z,t/2}(s) \leq L(s) \mbox{ for all }s \in [0, t/2]) &\leq \P(B^{br}_{x,z,t/2}(s) \leq L \mbox{ for all }s \in [0, t/2]) \nonumber \\
&= 1 - \P(B^{br}_{x,z,t/2}(s) \geq L \mbox{ for some }s \in [0, t/2]) \nonumber \\
&= 1 - \exp \bigg(- \frac{4(L-x)(L-z)}{t} \bigg) \nonumber \\
&\leq \frac{4(L-x)L}{t}.
\end{align}
Combining (\ref{br1}) and (\ref{br2}), we get
\begin{equation}\label{br5}
\P(0 \leq B^{br}_{x,z,t/2}(s) \leq L(s) \mbox{ for all }s \in [0, t/2]) \leq \frac{4L}{t} \min\{x, L-x\} \leq \frac{CL^2}{t} \sin \bigg( \frac{\pi x}{L} \bigg).
\end{equation}
If $0 \leq y \leq K/2$ and $0 \leq z \leq L$, then using the same reasoning as in (\ref{br1}),
\begin{align}\label{br3}
\P(0 \leq B^{br}_{z,y,t/2}(s) \leq L(t/2 + s) \mbox{ for all }s \in [0, t/2]) &\leq \P(B^{br}_{z,y,t/2}(s) \geq 0 \mbox{ for all } s \in [0, t/2]) \nonumber \\
&\leq \frac{4yL}{t}.
\end{align}
If $K/2 \leq y \leq K$, then by (\ref{bbr1}) with $a = K + mt/2$ and $b = -m$,
\begin{align} \label{br4}
&\P(0 \leq B^{br}_{z,y,t/2}(s) \leq L(t/2 + s) \mbox{ for all }s \in [0, t/2]) \nonumber \\
&\qquad \leq \P(B^{br}_{z,y,t/2}(s) \leq K + m(t/2 - s) \mbox{ for all }s \in [0, t/2]) \nonumber \\
&\qquad = 1 - \P(B^{br}_{z,y,t/2}(s) \geq K + m(t/2 - s) \mbox{ for some }s \in [0, t/2]) \nonumber \\
&\qquad = 1 - \exp \bigg( \frac{4(K + mt/2 - z)(K - y)}{t} \bigg). \nonumber \\
&\qquad \leq \frac{4(K + mt/2)(K-y)}{t}.
\end{align}
From (\ref{br3}) and (\ref{br4}) and the assumption that $K + mt/2 \leq 2L$, we get
\begin{equation}\label{br6}
\P(0 \leq B^{br}_{z, y, t/2}(s) \leq L(t/2 + s) \mbox{ for all }s \in [0, t/2]) \leq \frac{8L}{t} \min\{y, K-y\} \leq \frac{CL^2}{t} \sin \bigg( \frac{\pi y}{K} \bigg).
\end{equation}
By (\ref{brintg}), (\ref{br5}), and (\ref{br6}),
\begin{align}
\P(0 \leq B^{br}_{x,y,t}(s) \leq L(s) \mbox{ for all }s \in [0,t]) &\leq \frac{CL^4}{t^2} \sin \bigg( \frac{\pi x}{L} \bigg) \sin \bigg( \frac{\pi y}{K} \bigg) \int_0^{L(t/2)} g(z) \: dz \nonumber \\
&\leq \frac{CL^4}{t^2} \sin \bigg( \frac{\pi x}{L} \bigg) \sin \bigg( \frac{\pi y}{K} \bigg). \nonumber
\end{align}
The lemma follows by combining this result with (\ref{qtbrint}).
\end{proof}

\subsection{A curved right boundary}

We now consider the more general case in which the right boundary may change over time, which was studied in detail in \cite{haro}.  In \cite{haro}, Harris and Roberts considered branching Brownian motion restricted to stay between $f(s) - L(s)$ and $f(s) + L(s)$, which is equivalent to our setting when both $f(s)$ and $L(s)$ are set equal to what we have denoted by $L(s)/2$.  Assume that $s \mapsto L(s)$ is twice continuously differentiable.

Fix a point $x$ such that $0 < x < L(0)$.  Following the analysis in \cite{haro}, let $(\xi_t)_{t \geq 0}$ be a standard Brownian motion started at $x$, and define
\begin{align*}
G(s) &= \exp \left( \frac{1}{2} \int_0^s L'(u) \: d\xi_u  - \frac18 \int_0^s L'(u)^2 \: du + \int_0^s \frac{\pi^2}{2L(u)^2} \: du \right) \\ &\qquad \times \exp \left( \frac{L'(s)}{2L(s)} (\xi_s -L(s)/2)^2 -\int_0^s \left( \frac{L''(u)}{2 L(u)} (\xi_u - L(u)/2)^2 + \frac{L'(u)}{2L(u)} \right) du  \right).
\end{align*}
Also, define
\begin{equation}\label{Vdef}
V(s) = G(s) \sin \left(\frac{\pi \xi_s}{L(s)}\right) {\bf 1}_{\{ 0 < \xi_u < L(u) \, \forall u \le s\}}.
\end{equation}
It is shown in \cite{haro} using It\^o's Formula (see Lemma 4.2 of \cite{haro} and the discussion immediately following that result) that the process $(V(s), s \geq 0)$ is a martingale.

We now write $G(s)$ as a product of three terms $G(s) = A(s) B(s) C(s)$ as follows:
\begin{align*}
A(s) &=   \exp \left( \frac12 \int_0^s L'(u) \: d\xi_u  - \frac18 \int_0^s L'(u)^2 \: du \right)\\
B(s) &=  \exp \left(  \int_0^s \frac{\pi^2}{2L(u)^2} \: du -  \int_0^s \frac{L'(u)}{2L(u)} \: du  \right) \\
C(s) &=  \exp \left(   \frac{L'(s)}{2L(s)} (\xi_s -L(s)/2)^2 -\int_0^s \frac{L''(u)}{2L(u)} (\xi_u - L(u)/2)^2 \: du  \right).
\end{align*}
This leads to the following result about the expectation of $Z(s)$.

\begin{Lemma}\label{EZs}
Suppose initially there is a single particle at $x$.  Then $$\E[Z(s)] = e^{\sqrt{2} x}B(s)^{-1} \E[V(s) A(s)^{-1} C(s)^{-1}].$$
\end{Lemma}

\begin{proof}
Recall that $(\xi_t)_{t \geq 0}$ is standard Brownian motion with $\xi_0 = x$.  By the well-known Many-to-One Lemma for branching Brownian motion (see, for example, equation (3) of \cite{hh07}),
$$\E[Z(s)] = e^s \E\left[  e^{\sqrt{2} (\xi_s -\sqrt{2}s) } \sin \bigg( \frac{\pi (\xi_s -\sqrt{2} s)}{L(s)} \bigg) {\bf 1}_{\{0 < \xi_u - \sqrt{2}u < L(u) \: \forall u \le s\}} \right].$$  Using Girsanov's Theorem to relate Brownian motion with drift to standard Brownian motion,
\begin{align}
\E[Z(s)] &= e^s \E\left[ e^{-s - \sqrt{2}(\xi_s - x)} \cdot e^{\sqrt{2} \xi_s} \sin \bigg( \frac{\pi \xi_s}{L(s)} \bigg) {\bf 1}_{\{0 < \xi_u < L(u) \: \forall u \le s\}} \right] \nonumber \\
&= e^{\sqrt{2} x} \E \left[ \sin \bigg( \frac{\pi \xi_s}{L(s)} \bigg) {\bf 1}_{\{0 < \xi_u < L(u) \:  \forall u \le s\}} \right] \nonumber \\
&= e^{\sqrt{2} x} \E \left[ \frac{V(s)}{G(s)} \right] \nonumber \\
&= e^{\sqrt{2} x} B(s)^{-1} \E[V(s) A(s)^{-1} C(s)^{-1}], \nonumber
\end{align}
as claimed.
\end{proof}

\section{The case $L(s) = c(t-s)^{1/3}$}\label{stripsec2}

Fix any time $t > 0$, and for $0 \leq s \leq t$, define $$L(s) = c (t-s)^{1/3},$$ where $c$ was defined in (\ref{cdef}).
This right boundary was previously considered by Kesten \cite{kesten}.
Note that for $0 < s < t$, $$L'(s) = - \frac{c}{3} (t - s)^{-2/3}$$ and $$L''(s) = - \frac{2c}{9} (t - s)^{-5/3}.$$  Also, a straightforward calculation gives $$B(s)^{-1} = \exp \bigg( -(3 \pi^2)^{1/3}\big(t^{1/3} - (t - s)^{1/3}\big) \bigg) \bigg( \frac{t-s}{t} \bigg)^{1/6}.$$  We consider in this section branching Brownian motion with drift $-\sqrt{2}$ in which particles are killed if they reach $0$ or $L(s)$ at time $s$.  All particles will be killed by time $t$ because $L(t) = 0$.  We define $X_i(s)$, $N(s)$, and $Z(s)$ as in Section \ref{stripsec}.

\subsection{Estimating $\E[Z(s)]$}

In this section, we will estimate $\E[Z(s)]$ when $0 < s < t$.  In view of Lemma \ref{EZs}, this will require bounds on $A(s)$ and $C(s)$, which we present in Lemmas \ref{L: C} and \ref{L: A} below.  Note that the constants $c_1, \dots, c_6$ in these lemmas and in Proposition \ref{EZbound} do not depend on the initial position $x$ of the Brownian motion $(\xi_t)_{t \geq 0}$.

\begin{Lemma}\label{L: C}
There exist positive constants $c_1$ and $c_2$ such that for all $s \in (0, t)$, almost surely on the event $\{0 < \xi_u < L(u) \: \forall u \le s\}$ we have $$\exp (-c_1(t-s)^{-1/3}) \le C(s) \le \exp(c_2 (t-s)^{-1/3}).$$
\end{Lemma}

\begin{proof}
On the event $\{0 < \xi_u < L(u) \: \forall u \le s\}$, we have
\begin{align} \label{Cupper}
C(s) & \le  \exp \left(   \int_0^s \left|\frac{L''(u)}{2 L(u)} (\xi_u - L(u)/2)^2  \right|du \right) \nonumber \\
&\le \exp \left( \int_0^s \left| \frac{L''(u) L(u)}{8} \right| du \right) \nonumber \\
&= \exp \left( \frac{c^2}{36}  \int_0^s  (t-u)^{-4/3} \: du  \right) \nonumber \\
&\le \exp \left( \frac{c^2}{12} (t-s)^{-1/3}  \right).
\end{align}
On the other hand, on the event $\{0 < \xi_u < L(u) \: \forall u \le s\}$,
\begin{align} \label{Clower}
C(s) & \ge  \exp \left(  \frac{L'(s)}{2L(s) } (\xi_s -L(s)/2)^2  \right) \nonumber \\
&\ge \exp \left(-\frac{c^2}{24} (t-s)^{-1/3} \right).
\end{align}
The result follows from (\ref{Cupper}) and (\ref{Clower}).
\end{proof}

\begin{Lemma}\label{L: A}
There exist positive constants $c_3$ and $c_4$ such that for all $s \in (0, t)$, almost surely on the event $\{0 < \xi_u < L(u) \: \forall u \le s\}$ we have $$\exp(- c_3 (t-s)^{-1/3}) \le A(s) \le \exp (c_4 (t-s)^{-1/3}).$$
\end{Lemma}

\begin{proof}
Observe that $$\int_0^s L'(u)^2 \: du = \frac{c^2}{3} \left( (t - s)^{-1/3} - t^{-1/3} \right) \leq \frac{c^2}{3} (t - s)^{-1/3}.$$  Therefore, $$\exp \left( \frac{1}{2} \int_0^s L'(u) \: d\xi_u \right) \exp \left( - \frac{c^2}{24} (t-s)^{-1/3} \right) \le A(s) \le \exp \left( \frac{1}{2} \int_0^s L'(u) \: d\xi_u \right),$$ so it suffices to prove the result with $\exp(\frac{1}{2} \int_0^s L'(u) \: d\xi_u)$ in place of $A(s)$.

Using the Integration by Parts Formula and the fact that $L'$ has finite variation,
\begin{align*}
\int_0^s L'(u) \: d\xi_u &= L'(s)\xi_s -L'(0)\xi_0  -\int_0^s L''(u) \xi_u \: du.
\end{align*}
On the event $\{0 < \xi_u < L(u) \: \forall u \le s\}$, we have $0 \le - L'(s)\xi_s \le \frac{c^2}{3}(t-s)^{-1/3}$, which is also valid for $s = 0$, and $$0 \le - \int_0^s L''(u) \xi_u \: du \leq \frac{2c^2}{9} \int_0^s (t-u)^{-4/3} \: du \le \frac{2c^2}{3} (t-s)^{-1/3}.$$  These inequalities yield the conclusion.
\end{proof}

\begin{Prop} \label{EZbound}
There exist positive constants $c_5$ and $c_6$ such that for all $s \in (0, t)$,
$$Z(0) B(s)^{-1} \exp(- c_5 (t-s)^{-1/3})  \le \E[Z(s)] \le Z(0) B(s)^{-1} \exp(c_6 (t-s)^{-1/3}).$$
\end{Prop}

\begin{proof}
First, suppose that initially there is a single particle at $x$ with $0 < x < L(0)$.  Recall the definition of $V(s)$ from (\ref{Vdef}).  Because $V(s) = 0$ outside of the event $\{0 < \xi_u < L(u) \: \forall u \le s\}$, it follows from Lemmas \ref{EZs}, \ref{L: C}, and \ref{L: A} that there are constants $c_7$ and $c_8$ such that
$$e^{\sqrt{2} x} B(s)^{-1} \E[V(s)] \exp(-c_7 (t-s)^{-1/3}) \leq \E[Z(s)] \leq e^{\sqrt{2} x} B(s)^{-1} \E[V(s)] \exp(c_8 (t-s)^{-1/3}).$$  Because $(V(s), s \geq 0)$ is a martingale, $$e^{\sqrt{2}x}\E[V(s)] = e^{\sqrt{2}x}V(0) = e^{\sqrt{2} x} G(0) \sin \bigg( \frac{\pi x}{L(0)} \bigg) = Z(0) G(0).$$  The result when there is initially a single particle at $x$ follows because $$1 \geq G(0) = \exp \bigg( \frac{L'(0)}{2L(0)} (\xi_0 - L(0)/2)^2 \bigg) \geq \exp \bigg( \frac{L'(0) L(0)}{8} \bigg) = \exp \bigg( - \frac{c^2}{24} t^{-1/3} \bigg).$$  Because $B(s)$ and the constants $c_5$ and $c_6$ do not depend on the position $x$ of the initial particle, the result follows for general initial configurations by summing over the particles.
\end{proof}

\begin{Cor}\label{condexpZ}
Let $({\cal F}_u, u \geq 0)$ be the natural filtration associated with the branching Brownian motion.  Let $0 < r < s < t$.  Let
\begin{align}
B_r(s) &= \exp \bigg( \int_r^s \frac{\pi^2}{2 L(u)^2} \: du - \int_r^s \frac{L'(u)}{2 L(u)} \: du \bigg) \nonumber \\
&= \exp \bigg(  (3 \pi^2)^{1/3} \big((t-r)^{1/3} - (t-s)^{1/3} \big) \bigg) \bigg( \frac{t-r}{t-s} \bigg)^{1/6}. \nonumber
\end{align}
Then $$Z(r) B_r(s)^{-1} \exp(-c_5 (t - s)^{-1/3}) \leq \E[Z(s)|{\cal F}_r] \leq Z(r) B_r(s)^{-1} \exp(c_6(t - s)^{-1/3}),$$ where $c_5$ and $c_6$ are the constants from Proposition \ref{EZbound}.
\end{Cor}

\begin{proof}
Apply the Markov Property at time $r$, and then apply Proposition \ref{EZbound} with $t^* = t - r$ and $L^*(u) = c(t^* - u)^{1/3} = c(t - r - u)^{1/3} = L(u + r)$.
\end{proof}

\subsection{Bounding the density}

We now use the estimate of $\E[Z(s)]$ from Proposition \ref{EZbound} to obtain bounds on the density.  For $0 \leq r < s < t$, let $q_{r,s}(x,y)$ represent the density of particles at time $s$ that are descended from a particle at the location $x$ at time $r$.  That is, if $A$ is a Borel subset of $(0, L(s))$, then the expected number of particles in $A$ at time $s$ descended from the particle which is at $x$ at time $r$ is $$\int_A q_{r,s}(x,y) \: dy.$$  Note that $q_s(x,y) = q_{0,s}(x,y)$. For $x,y>0$ and $0 \le r\le s \le t$, let $$\psi_{r,s}(x,y) = \frac1{L(s)} e^{-(3 \pi^2)^{1/3}((t-r)^{1/3} - (t-s)^{1/3})} \bigg( \frac{t-s}{t-r} \bigg)^{1/6} e^{\sqrt{2} x} \sin \bigg( \frac{\pi x}{L(r)} \bigg) e^{-\sqrt{2} y} \sin \bigg( \frac{\pi y}{L(s)} \bigg).$$
This expression becomes simpler if we view the process from time $t$, as we get $$\psi_{t-u,t-v}(x,y) = \frac1{c} e^{-(3 \pi^2)^{1/3} (u^{1/3} - v^{1/3})}\left(\frac{1}{uv}\right)^{1/6} e^{\sqrt{2} x} \sin \bigg( \frac{\pi x}{cu^{1/3}} \bigg) e^{-\sqrt{2} y} \sin \bigg( \frac{\pi y}{cv^{1/3}} \bigg).$$

\begin{Prop}\label{densityprop}
Fix a positive constant $b$.  There exists a constant $A > 0$ and positive constants $C'$ and $C''$, with $C''$ depending on $b$, such that if $r + L(r)^2 \leq s \leq t - A$, then
\begin{equation}\label{qrsxyl}
q_{r,s}(x,y) \geq C' \psi_{r,s}(x,y),
\end{equation}
and if $r + bL(r)^2 \leq s \leq t - A$, then
\begin{equation}\label{qrsxyu}
q_{r,s}(x,y) \leq C'' \psi_{r,s}(x,y).
\end{equation}
\end{Prop}

\begin{proof}
Let $\E_{r,x}$ denote expectation for the process starting from a single particle at $x$ at time $r$.  Note that if $r < u < s$, then
\begin{equation}\label{qrsu}
q_{r,s}(x,y) = \int_0^{L(u)} q_{r,u}(x,z) q_{u,s}(z,y) \: dz.
\end{equation}

We first prove the upper bound.  We may assume $b \leq 1$.  Assume $r + bL(r)^2 \leq s \leq t - A$.  Let $u = s - { b}L(s)^2$.  Note that $u > r$ because $L(s) < L(r)$.   Let $m = -2L'(s) = (2c/3)(t-s)^{-2/3}$.  For $u \leq v \leq s$, let
\begin{displaymath}
{\hat L}(v) = \left\{
\begin{array}{ll}  L(u) & \mbox{ if }u \leq v \leq s - m^{-1}(L(u) - L(s))  \\
L(s) + m(s-v) & \mbox{ if } s - m^{-1}(L(u)-L(s)) \leq v \leq s.
\end{array} \right.
\end{displaymath}

\begin{figure}\centering
\includegraphics[scale=1]{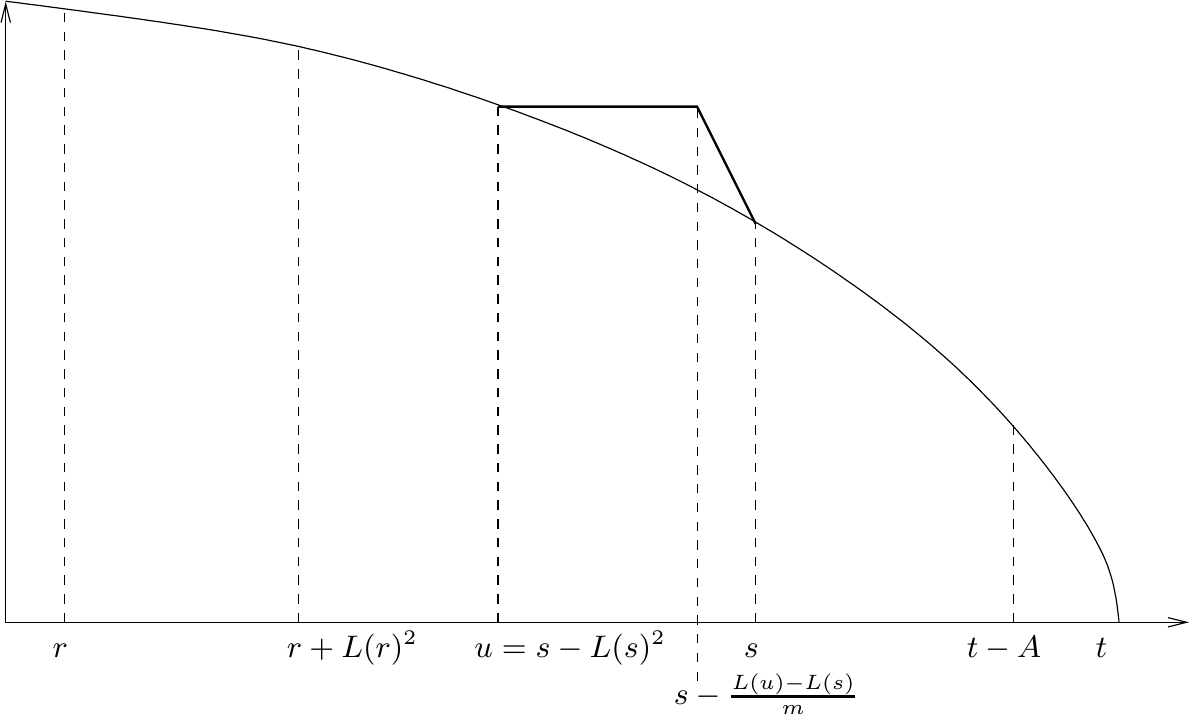}
\caption{ the function $\hat L$ }
\label{F:branching}
\end{figure}

Note that ${\hat L}(v) \geq L(v)$ for all $v \in [u, s]$.  Therefore, if we define ${\hat q}_{u,s}(z,y)$ in the same way as $q_{u,s}(z,y)$, except that for $v \in [u,s]$, particles are killed when they reach ${\hat L}(v)$ instead of when they reach $L(v)$, then
\begin{equation}\label{qqhat}
q_{u,s}(z,y) \leq {\hat q}_{u,s}(z,y).
\end{equation}

We now wish to apply Lemma \ref{pllem} with $K = L(s)$, $L = L(u)$ and $t = s - u$.  We need to check first that $L(s)+m(s-u)/2\le 2L(u)$ and second that $m^{-1}(L(u)-L(s))\le (s-u)/2$.
For the first condition, as long as $A$ is chosen to be large enough that $L(t - A) \geq c^3/3$, we have
$$L(s) + \frac{m(s-u)}{2} = L(s) + \frac{mbL(s)^2}{2} = L(s) + \frac{b c^3}{3} \leq 2L(s) \leq 2L(u).$$
The second condition also holds because $$m^{-1}(L(u) - L(s)) \leq m^{-1}|L'(s)|(s - u) = \frac{s-u}{2}.$$
Therefore, by Lemma \ref{pllem},
\begin{equation}\label{hatq}
{\hat q}_{u,s}(z,y) \leq \frac{C L(u)^4}{(b L(s)^2)^{5/2}} e^{\sqrt{2} z} \sin \bigg( \frac{\pi z}{L(u)} \bigg) e^{-\sqrt{2} y} \sin \bigg( \frac{\pi y}{L(s)} \bigg).
\end{equation}
Note that
\begin{equation}\label{LuLs}
L(u) - L(s) \leq -L'(s)(s - u) = \frac{b c^3}{3}.
\end{equation}
Therefore, if $A$ is large enough that $L(t - A) \geq c^3/3$, then $L(u) \leq 2 L(s)$,
so combining (\ref{qrsu}), (\ref{qqhat}), (\ref{hatq}), we get
\begin{align}
q_{r,s}(x,y) &\leq \frac{C}{L(s)} e^{-\sqrt{2} y} \sin \bigg( \frac{\pi y}{L(s)} \bigg) \int_0^{L(u)} e^{\sqrt{2} z} \sin \bigg( \frac{\pi z}{L(u)} \bigg) q_{r,u}(x,z) \: dz \nonumber \\
&= \frac{C}{L(s)} e^{-\sqrt{2} y} \sin \bigg( \frac{\pi y}{L(s)} \bigg) \E_{r,x}[Z(u)]. \nonumber
\end{align}
Therefore, using Corollary \ref{condexpZ} to bound $\E_{r,x}[Z(u)]$,
$$q_{r,s}(x,y) \leq \frac{C}{L(s)} e^{-(3 \pi^2)^{1/3}((t - r)^{1/3} - (t - u)^{1/3})} \bigg( \frac{t-u}{t-r} \bigg)^{1/6} e^{\sqrt{2} x} \sin \bigg( \frac{\pi x}{L(r)} \bigg) e^{-\sqrt{2} y} \sin \bigg( \frac{\pi y}{L(s)} \bigg).$$
The upper bound (\ref{qrsxyu}) now follows because $(t - u)^{1/3} \leq (t - s)^{1/3} + bc^2/3$ by (\ref{LuLs}) and $t - u = (t - s) + (s - u) \leq C(t - s)$.

We next prove the lower bound.  Assume that $r + L(r)^2 \leq s \leq t - A$.  Let $u = s - L(s)^2/2$.  Note that $u > r$ because $L(s) < L(r)$.  For $0 \leq z \leq L(s)$, define ${\tilde q}_{u,s}(z,y)$ in the same way as $q_{u,s}(z,y)$ except that for $v \in [u,s]$, particles are killed when they reach $L(s)$ instead of when they reach $L(v)$.  Then
\begin{equation}\label{qqtild}
q_{u,s}(z,y) \geq {\tilde q}_{u,s}(z,y).
\end{equation}
By Lemma \ref{stripdensity}, if $0 \leq z \leq L(s)$, then because $$\sum_{n=2}^{\infty} n^2 e^{-\pi^2 (n^2 - 1)(s-u)/2L(s)^2} = \sum_{n=2}^{\infty} n^2 e^{-\pi^2 (n^2 - 1)/4} < 1,$$ we have
\begin{equation}\label{tildq}
{\tilde q}_{u,s}(z,y)\geq \frac{C}{L(s)} e^{-\pi^2 (s-u)/2L(s)^2} e^{\sqrt{2} z} \sin \bigg( \frac{\pi z}{L(s)} \bigg) e^{-\sqrt{2} y} \sin \bigg( \frac{\pi y}{L(s)} \bigg).
\end{equation}
By (\ref{qrsu}), (\ref{qqtild}), and (\ref{tildq}),
$$q_{r,s}(x,y) \geq \frac{C}{L(s)} e^{-\sqrt{2} y} \sin \bigg( \frac{\pi y}{L(s)} \bigg) \int_0^{L(s)} e^{\sqrt{2} z} \sin \bigg( \frac{\pi z}{L(s)} \bigg) q_{r,u}(x,z) \: dz.$$  Using (\ref{LuLs}) with $b=1/2$, we get $L(u) - L(s) \leq c^3/6$.  Therefore, there is a positive constant $C$ such that $\sin(\pi z/L(s)) \geq C \sin(\pi z/L(u))$ for all $z \leq L(u) - c^3$.
It follows that
\begin{equation}\label{qlow1}
q_{r,s}(x,y) \geq \frac{C}{L(s)} e^{-\sqrt{2} y} \sin \bigg( \frac{\pi y}{L(s)} \bigg) \bigg( \E_{r,x}[Z(u)] - \int_{L(u) - c^3}^{L(u)} e^{\sqrt{2} z} \sin \bigg( \frac{\pi z}{L(u)} \bigg) q_{r,u}(x,z) \: dz \bigg).
\end{equation}
By Corollary \ref{condexpZ},
\begin{equation} \label{qlow2}
\E_{r,x}[Z(u)] \geq C e^{-(3 \pi^2)^{1/3}((t-r)^{1/3} - (t-u)^{1/3})} \bigg( \frac{t-u}{t-r} \bigg)^{1/6} e^{\sqrt{2} x} \sin \bigg( \frac{\pi x}{L(r)} \bigg).
\end{equation}
Also, because $u - r = s - L(s)^2/2 - r \geq L(r)^2 - L(s)^2/2 \geq L(r)^2/2$, we can apply the upper bound (\ref{qrsxyu}) to get
\begin{align}\label{qlow3}
&\int_{L(u) - c^3}^{L(u)} e^{\sqrt{2} z} \sin \bigg( \frac{\pi z}{L(u)} \bigg) q_{r,u}(x,z) \: dz \nonumber \\
&\qquad \leq \frac{C}{L(u)} e^{-(3 \pi^2)^{1/3}((t - r)^{1/3} - (t-u)^{1/3})} \bigg( \frac{t-u}{t-r} \bigg)^{1/6} e^{\sqrt{2} x} \sin \bigg( \frac{\pi x}{L(r)} \bigg) \int_{L(u) - c^3}^{L(u)} \sin\bigg( \frac{\pi z}{L(u)} \bigg)^2 \: dz \nonumber \\ 
&\qquad \leq \frac{C}{L(u)^3} e^{-(3 \pi^2)^{1/3}((t - r)^{1/3} - (t-u)^{1/3})} \bigg( \frac{t-u}{t-r} \bigg)^{1/6} e^{\sqrt{2} x} \sin \bigg( \frac{\pi x}{L(r)} \bigg).
\end{align}
Choosing $A$ sufficiently large, the lower bound (\ref{qrsxyl}) now follows from (\ref{qlow1}), (\ref{qlow2}), (\ref{qlow3}), and the fact that $t - u \geq t - s$.
\end{proof}

\subsection{Particles hitting the right boundary}

For $0 \leq s < u \leq t$, let $R_{s,u}$ denote the number of particles that are killed at $L(r)$ for some $r \in [s, u]$.  Let $\E_{s,x}$ denote expectation for the process started from a single particle at $x$ at time $s$.

\begin{Lemma}\label{ssLu2}
If $0 \leq s < u < t$, then
$$\E_{s,x}[R_{s, u}] \leq \frac{x e^{\sqrt{2} x} e^{-\sqrt{2} L(u)}}{L(u)}.$$
\end{Lemma}

\begin{proof}
For branching Brownian motion with absorption only at the origin, if we define $$M(s) = \sum_{i=1}^{N(s)} X_i(s) e^{\sqrt{2} X_i(s)},$$ then it is well-known (see, for example, Lemma 2 of \cite{hh07}) that the process $(M(s), s \geq 0)$ is a martingale.  Now, for $u \in [s, t]$, let
\begin{equation}\label{Msudef}
M_s(u) = \sum_{i=1}^{N(u)} X_i(u) e^{\sqrt{2} X_i(u)} + L(u) e^{\sqrt{2} L(u)} R_{s,u}.
\end{equation}
We claim that the process $(M_s(u), s \leq u \leq t)$ is a supermartingale for branching Brownian motion with killing both at the origin and at the right boundary $L(\cdot).$  To see this, observe that because the process $(M(s), s \geq 0)$ is a martingale when there is no killing at the right boundary, this process would still be a martingale if particles were stopped, but not killed, upon reaching the right boundary.  Because the function $u \mapsto L(u)$ is decreasing and because $x \mapsto x e^{\sqrt{2}x}$ is increasing, 
the process becomes a supermartingale if particles, after hitting the right boundary, follow the right boundary until time $t$.  This is the process defined in (\ref{Msudef}) because there will be $R_{s,u}$ particles at $L(u)$ at time $u$.

Because the process defined in (\ref{Msudef}) is a supermartingale, we have $$xe^{\sqrt{2} x} = \E_{s,x}[M_s(s)] \geq \E_{s,x}[M_s(u)] \geq \E_{s,x}[L(u) e^{\sqrt{2} L(u)} R_{s,u}] = L(u) e^{\sqrt{2} L(u)} \E_{s,x}[R_{s,u}].$$  The result follows.
\end{proof}

\begin{Lemma}\label{sxuu1}
There is a constant $A > 0$ such that for all $s$, $u$, and $x$ such that $s \geq 0$, $0 < x < L(s)$, and $s + L(s)^2 \leq u \leq t - A$, we have $$\E_{s,x}[R_{u, u+1}] \asymp \frac{1}{L(u)^2} e^{-(3 \pi^2)^{1/3} (t-s)^{1/3}} \bigg( \frac{t-u}{t-s} \bigg)^{1/6} e^{\sqrt{2} x} \sin \bigg( \frac{\pi x}{L(s)} \bigg).$$
\end{Lemma}

\begin{proof}
We adapt ideas from the proofs of Lemma 15 and Proposition 16 in \cite{bbs}.  By applying the Markov property at time $u$, we get
\begin{equation}\label{Marku}
\E_{s,x}[R_{u, u+1}] = \int_0^{L(u)} q_{s,u}(x,y) \E_{u,y}[R_{u, u+1}] \: dy.
\end{equation}
Let $(\xi_r)_{r \geq 0}$ be standard Brownian motion with $\xi_0 = 0$.  Because a particle at time $u$ will have on average $e$ descendants at time $u+1$ if no particles are killed, the expectation $\E_{u,y}[R_{u,u+1}]$ is bounded above by $e$ times the probability that a particle started from $y$ at time $u$ is to the right of $L(u+1)$ at some time before time $u+1$.  Therefore, it follows from the Reflection Principle and the inequality $$\int_z^{\infty} e^{-x^2/2} \: dx \leq z^{-1} e^{-z^2/2}$$ that if $y \leq L(u+1)$, then
\begin{align}
\E_{u,y}[R_{u, u+1}]  &\leq e \P\big( \max_{0 \leq r \leq 1} (\xi_r - \sqrt{2} r) \geq L(u+1) - y \big) \nonumber \\
&\leq 2e \P(\xi_1 \geq L(u+1) - y) \nonumber \\
&\leq \frac{C}{L(u+1) - y} e^{-(L(u+1) - y)^2/2}. \nonumber
\end{align}
Therefore, letting $\alpha = L(u) - L(u+1)$ and requiring $A$ to be large enough that $L(t-A+1) > 1$,
we have (using the change of variable $z=L(u)-y)$
\begin{align}\label{siny1}
&\int_0^{L(u) - \alpha - 1} e^{-\sqrt{2} y} \sin \bigg( \frac{\pi y}{L(u)} \bigg) \E_{u,y}[R_{u, u+1}] \: dy \nonumber \\
&\qquad \leq C \int_0^{L(u) - \alpha - 1} e^{-\sqrt{2} y} \sin \bigg( \frac{\pi y}{L(u)} \bigg) \frac{1}{L(u+1) - y} e^{-(L(u+1) - y)^2/2} \: dy \nonumber \\
&\qquad \leq C e^{-\sqrt{2} L(u)} \int_{\alpha + 1}^{L(u)} e^{\sqrt{2} z} \cdot \frac{\pi z}{L(u)} \cdot \frac{1}{z - \alpha} e^{-(z - \alpha)^2/2} \: dz \nonumber \\
&\qquad \leq \frac{C e^{-\sqrt{2} L(u)}}{L(u)}.
\end{align}
Using the bound $\E_{u,y}[R_{u, u+1}] \leq e$, we get
\begin{equation}\label{siny2}
\int_{L(u) - \alpha - 1}^{L(u)} e^{-\sqrt{2} y} \sin \bigg( \frac{\pi y}{L(u)} \bigg) \E_{u,y}[R_{u, u+1}] \: dy \leq \frac{C e^{-\sqrt{2} L(u)}}{L(u)}.
\end{equation}
Combining (\ref{siny1}) and (\ref{siny2}) with (\ref{Marku}) and Proposition \ref{densityprop}, and using the fact that $e^{-\sqrt{2} L(u)} = e^{-(3 \pi^2)^{1/3} (t - u)^{1/3}}$, we get, for $A$ large enough,
$$\E_{s,x}[R_{u, u+1}] \leq \frac{C''}{L(u)^2} e^{-(3 \pi^2)^{1/3}(t - s)^{1/3}} \bigg( \frac{t-u}{t-s} \bigg)^{1/6} e^{\sqrt{2} x} \sin \bigg( \frac{\pi x}{L(s)} \bigg),$$ which is the upper bound in the statement of the lemma.

Next, observe that for $y \in [L(u) - 1, L(u)]$, we have $$\E_{u,y}[R_{u, u+1}] \geq \P(\xi_1 - \sqrt{2} \geq L(u+1) - y) \geq \P(\xi_1 \geq 1 + \sqrt{2}) \geq C.$$  Thus, by (\ref{Marku}) and Proposition \ref{densityprop},
\begin{align}
\E_{s,x}[R_{u,u+1}] &\geq \frac{C'}{L(u)} e^{-(3 \pi^2)^{1/3}((t-s)^{1/3} - (t-u)^{1/3})} \bigg( \frac{t-u}{t-s} \bigg)^{1/6}  \nonumber \\
&\qquad \times e^{\sqrt{2} x} \sin \bigg( \frac{\pi x}{L(s)} \bigg) \int_{L(u) - 1}^{L(u)} e^{-\sqrt{2} y} \sin \bigg( \frac{\pi y}{L(u)} \bigg) \: dy \nonumber \\
&\geq  \frac{C'}{L(u)^2} e^{-(3 \pi^2)^{1/3}(t - s)^{1/3}} \bigg( \frac{t-u}{t-s} \bigg)^{1/6} e^{\sqrt{2} x} \sin \bigg( \frac{\pi x}{L(s)} \bigg), \nonumber
\end{align}
which gives the required lower bound.
\end{proof}

\begin{Lemma} \label{L:sumRj}
There is a constant $A_0 > 0$ and positive constants $C'$ and $C''$ such that if $0 \leq s \leq t - A_0$ and $0 < x < L(s)$, then
\begin{equation}\label{sumRj:ub}
C' h(s,x) \leq \E_{s,x}[R_{s,t}] \leq C''(h(s,x) + j(s,x)),
\end{equation}
where
\begin{equation}\label{hsdef}
h(s,x) = e^{\sqrt{2}x} \sin\left(\frac{\pi x}{L(s)}\right) (t-s)^{1/3} \exp(- (3\pi^2(t-s))^{1/3})
\end{equation}
and
$$j(s,x) = x e^{\sqrt{2}x} (t-s)^{-1/3} \exp(- (3\pi^2(t-s))^{1/3}).$$  Also, if $0 < \alpha < \beta < 1$, then
\begin{equation}\label{sumRj2}
C' h(s,x) \leq \E_{s,x}[R_{s + \alpha(t-s), s + \beta(t-s)}] \leq C'' h(s,x),
\end{equation}
where the constants $C'$ and $C''$ depend on $\alpha$ and $\beta$.
\end{Lemma}

\begin{proof}
If $u = s + L(s)^2$, then for sufficiently large $A_0$, $$L(s) - L(u) \leq -L'(u)(u - s) = \frac{c^3}{3} \bigg( \frac{t-s}{t-u} \bigg)^{2/3} \leq C.$$  Therefore, by Lemma \ref{ssLu2}, using that $\sqrt{2} L(s) = (3 \pi^2(t-s))^{1/3}$,
\begin{equation}\label{EsxR1}
0 \leq \E_{s,x}[R_{s, s + L(s)^2}] \leq \frac{C x e^{\sqrt{2} x} e^{-\sqrt{2} L(s)}}{L(s)} \leq C j(s,x).
\end{equation}
We may choose $A_0$ to be large enough that $s + L(s)^2 \leq t - A - 1$ whenever $0 \leq s \leq t - A_0$, where $A$ is the constant from Lemma \ref{sxuu1}.  By Lemma \ref{sxuu1},
\begin{align}\label{EsxR2}
\E_{s,x}[R_{s+L(s)^2, t-A}] &\asymp \frac{e^{-(3 \pi^2)^{1/3}(t-s)^{1/3}}}{(t-s)^{1/6}} e^{\sqrt{2} x} \sin \bigg(\frac{\pi x}{L(s)} \bigg) \int_{s + L(s)^2}^{t-A} \frac{(t-u)^{1/6}}{L(u)^2} \: du \nonumber \\
&\asymp \frac{e^{-(3 \pi^2)^{1/3}(t-s)^{1/3}}}{(t-s)^{1/6}} e^{\sqrt{2} x} \sin \bigg(\frac{\pi x}{L(s)} \bigg) \int_{s + L(s)^2}^{t - A} \frac{1}{(t-u)^{1/2}} \: du \nonumber \\
&\asymp h(s,x).
\end{align}
Because particles branch at rate one, $\E_{s,x}[R_{t-A, t}]$ is at most $e^A$ times the expected number of particles between $0$ and $L(t-A)$ at time $t-A$.  Therefore, by Proposition \ref{densityprop},
\begin{align} \label{EsxR3}
\E_{s,x}[R_{t-A, t}] &\leq e^A \int_0^{L(t-A)} q_{s,t-A}(x,y) \: dy \nonumber \\
&\leq \frac{C e^{-(3 \pi^2)^{1/3}(t-s)^{1/3}}}{(t-s)^{1/6}} e^{\sqrt{2} x} \sin \bigg( \frac{\pi x}{L(s)} \bigg) \int_0^{L(t-A)} e^{-\sqrt{2} y} \sin \bigg( \frac{\pi y}{L(s)} \bigg) \: dy \nonumber \\
&\leq \frac{C h(s,x)}{(t-s)^{5/6}}.
\end{align}
The result (\ref{sumRj:ub}) follows from (\ref{EsxR1}), (\ref{EsxR2}), and (\ref{EsxR3}).  The result (\ref{sumRj2}) follows from the reasoning in (\ref{EsxR2}), using $s + \alpha(t-s)$ and $s + \beta(t-s)$ as the limits of integration.
\end{proof}

\begin{Lemma} \label{P:R2}
Let $0 < \alpha < \beta < 1$.  Let $A_0$ be the constant defined in Lemma \ref{L:sumRj}.  Then there exist positive constants $C'$ and $C''$ depending on $\alpha$ and $\beta$ such that if $t \geq A_0$ and $0 < x < L(0) - 1$, then $$\E_{0,x}[R_{\alpha t, \beta t}^2] \leq C h(0,x).$$
\end{Lemma}

\begin{proof}
The proof is similar to the proof of Proposition 18 in \cite{bbs}.  Throughout this proof, we write $R = R_{\alpha t, \beta t}$.  Note that $R^2 = R + 2Y$, where $Y$ is the number of distinct pairs of particles that reach $L(s)$ for some $s \in [\alpha t, \beta t]$.  A branching event at the location $y$ at time $s$ produces, on average, $(\E_{s,y}[R])^2$ pairs of particles that reach the right boundary and have their most recent common ancestor at time $s$.  Therefore, by Lemma \ref{L:sumRj}, we may write
\begin{align} \label{hsx0}
\E_{0,x}[R^2] &= \E_{0,x}[R] + 2 \int_0^{\beta t} \int_0^{L(s)} q_{0,s}(x,y) (\E_{s,y}[R])^2 \: dy \: ds \nonumber \\
& \leq \E_{0,x}[R] + C \int_0^{\beta t} \int_0^{L(s)} q_{0,s}(x,y)(h(s,y)^2 + j(s,y)^2) \: dy \: ds.
\end{align}
We bound separately the term involving $h(s,y)^2$ and the term involving $j(s,y)^2$.  We also treat separately the cases $s \leq L(0)^2$ and $s \geq L(0)^2$.

By Proposition \ref{densityprop} and (\ref{hsdef}),
\begin{align} \label{hsx1}
& \int_{L(0)^2}^{\beta t} \int_0^{L(s)} q_{0,s}(x,y) h(s,y)^2 \: dy \: ds \nonumber \\
& \qquad \le C e^{-(3 \pi^2)^{1/3}t^{1/3}} e^{\sqrt{2}x} \sin \bigg(\frac{\pi x}{L(0)} \bigg) \int_{L(0)^2}^{\beta t} \int_0^{L(s)} \frac{1}{L(s)} \bigg( \frac{t-s}{t} \bigg)^{1/6} e^{(3 \pi^2)^{1/3}(t-s)^{1/3}}  \nonumber \\
&\qquad \qquad \times e^{-\sqrt{2} y} \sin\left(\frac{\pi y}{L(s)}\right) \left\{(t-s)^{1/3} e^{-(3 \pi^2)^{1/3}(t-s)^{1/3}} e^{\sqrt{2}y} \sin\left(\frac{\pi y}{L(s)}\right)\right\}^2 \: dy \: ds \nonumber \\
& \qquad \le \frac{C e^{-(3 \pi^2)^{1/3}t^{1/3}}}{t^{1/6}} e^{\sqrt{2}x} \sin \bigg(\frac{\pi x}{L(0)}\bigg) \int_{L(0)^2}^{\beta t} e^{-(3 \pi^2)^{1/3}(t-s)^{1/3}} (t - s)^{1/2} \nonumber \\
&\qquad \qquad \times \int_0^{L(s)} e^{\sqrt{2}y} \sin\bigg(\frac{\pi y}{L(s)}\bigg)^3 \: dy \: ds \nonumber \\
& \qquad \le \frac{C e^{-(3 \pi^2)^{1/3}t^{1/3}}}{t^{1/6}} e^{\sqrt{2}x} \sin \bigg(\frac{\pi x}{L(0)}\bigg) \int_{L(0)^2}^{\beta t} e^{-(3 \pi^2)^{1/3}(t - s)^{1/3}} (t - s)^{1/2} \frac{e^{\sqrt{2} L(s)}}{L(s)^3} \: ds \nonumber \\
& \qquad \le \frac{C e^{-(3 \pi^2)^{1/3}t^{1/3}}}{t^{1/6}} e^{\sqrt{2}x} \sin \bigg(\frac{\pi x}{L(0)}\bigg) \int_{L(0)^2}^{\beta t} \frac{1}{(t - s)^{1/2}} \: ds \nonumber \\
& \qquad \leq C h(0,x).
\end{align}
A similar computation gives
\begin{align}\label{hsx2}
& \int_{L(0)^2}^{\beta t} \int_0^{L(s)} q_{0,s}(x,y) j(s,y)^2 \: dy \: ds \nonumber \\
& \qquad \le C e^{-(3 \pi^2)^{1/3}t^{1/3}} e^{\sqrt{2}x} \sin \bigg(\frac{\pi x}{L(0)} \bigg) \int_{L(0)^2}^{\beta t} \int_0^{L(s)} \frac{1}{L(s)} \bigg( \frac{t-s}{t} \bigg)^{1/6} e^{(3 \pi^2)^{1/3}(t - s)^{1/3}}  \nonumber \\
&\qquad \qquad \times e^{-\sqrt{2} y} \sin\left(\frac{\pi y}{L(s)}\right) \left\{(t-s)^{-1/3} e^{-(3 \pi^2)^{1/3}(t-s)^{1/3}} y e^{\sqrt{2}y} \right\}^2 \: dy \: ds \nonumber \\
&\qquad \le \frac{C e^{-(3 \pi^2)^{1/3}t^{1/3}}}{t^{1/6}} e^{\sqrt{2}x} \sin \bigg(\frac{\pi x}{L(0)}\bigg) \int_{L(0)^2}^{\beta t} e^{-(3 \pi^2)^{1/3}(t-s)^{1/3}} \frac{1}{(t-s)^{5/6}} \nonumber \\
&\qquad \qquad \times \int_0^{L(s)} e^{\sqrt{2} y} y^2 \sin \bigg( \frac{\pi y}{L(s)} \bigg) \: ds \nonumber \\
&\qquad \le  \frac{C e^{-(3 \pi^2)^{1/3}t^{1/3}}}{t^{1/6}} e^{\sqrt{2}x} \sin \bigg(\frac{\pi x}{L(0)}\bigg) \int_{L(0)^2}^{\beta t} \frac{1}{(t - s)^{1/2}} \: ds \nonumber \\
& \qquad \leq C h(0,x).
\end{align}

It remains to bound from above the two integrals between $0$ and $L(0)^2$. If $0 \leq s \leq L(0)^2$, then $t^{1/3} - (t-s)^{1/3} \leq C$, and $\sin(\pi y/L(s)) \leq C \sin (\pi y/L(0))$ for all $y \in [0, L(s)]$.  Also, because $q_{0,s}(x,y)$ is bounded above by the density that would be obtained if particles were killed at $L(0)$, rather than $L(r)$, for $r \in [0,s]$, Lemma \ref{mainqslem} implies that $$\int_0^{L(0)^2} q_{0,s}(x,y) \: ds \leq \frac{2 e^{\sqrt{2}(x-y)} x (L(0) - y)}{L(0)}.$$  Thus
\begin{align}
&\int_0^{L(0)^2} \int_0^{L(s)} q_{0,s}(x,y) h(s,y)^2 \: dy \: ds \nonumber \\
&\qquad \leq C \int_0^{L(0)^2} \int_0^{L(s)} q_{0,s}(x,y) \bigg\{ e^{\sqrt{2} y} \sin \bigg( \frac{\pi y}{L(s)} \bigg) (t - s)^{1/3} e^{-(3 \pi^2)^{1/3}(t-s)^{1/3}} \bigg\}^2 \: dy \: ds \nonumber \\
&\qquad \leq C e^{-2(3 \pi^2)^{1/3} t^{1/3}} t^{2/3} \int_0^{L(0)} e^{2 \sqrt{2} y} \sin \bigg( \frac{\pi y}{L(0)} \bigg)^2 \bigg( \int_0^{L(0)^2} q_{0,s}(x,y) \: ds \bigg) \: dy \nonumber \\
&\qquad \leq C x e^{\sqrt{2} x} e^{-2(3 \pi^2)^{1/3} t^{1/3}} t^{2/3} \int_0^{L(0)} e^{\sqrt{2} y} \sin \bigg( \frac{\pi y}{L(0)} \bigg)^2 \frac{L(0) - y}{L(0)} \: dy \nonumber \\
&\qquad \leq C x e^{\sqrt{2} x} e^{-2(3 \pi^2)^{1/3} t^{1/3}} t^{2/3} \cdot \frac{e^{\sqrt{2} L(0)}}{L(0)^3} \nonumber \\
&\qquad \leq C x e^{\sqrt{2} x} e^{-(3 \pi^2)^{1/3} t^{1/3}} t^{-1/3}. \nonumber
\end{align}
Because
\begin{equation}\label{585}
x t^{-1/3} \leq C t^{1/3} \sin(\pi x/L(0))
\end{equation}
when $0 < x < L(0) - 1$, it follows that
\begin{equation}\label{hsx3}
\int_0^{L(0)^2} \int_0^{L(s)} q_{0,s}(x,y) h(s,y)^2 \: dy \: ds \leq C h(0,x).
\end{equation}
Likewise, using that $y (t-s)^{-1/3} \leq C$ for $y \leq L(s)$, we get
\begin{align}\label{hsx4}
\int_0^{L(0)^2} \int_0^{L(s)} q_{0,s}(x,y) j(s,y)^2 \: dy \: ds
&\leq C \int_0^{L(0)^2} \int_0^{L(s)} q_{0,s}(x,y) \bigg\{ e^{\sqrt{2} y} e^{-(3 \pi^2)^{1/3}(t-s)^{1/3}} \bigg\}^2 \: dy \: ds \nonumber \\
&\leq C e^{-2(3 \pi^2)^{1/3} t^{1/3}} \int_0^{L(0)} e^{2 \sqrt{2} y} \bigg( \int_0^{L(0)^2} q_{0,s}(x,y) \: ds \bigg) \: dy \nonumber \\
&\leq C x e^{\sqrt{2} x} e^{-2(3 \pi^2)^{1/3} t^{1/3}} \int_0^{L(0)} e^{\sqrt{2} y} \cdot \frac{L(0) - y}{L(0)} \: dy \nonumber \\
&\leq C x e^{\sqrt{2} x} e^{-(3 \pi^2)^{1/3} t^{1/3}} t^{-1/3}. \nonumber \\
&\leq C h(0,x).
\end{align}
The result follows from (\ref{hsx0}), (\ref{hsx1}), (\ref{hsx2}), (\ref{hsx3}), (\ref{hsx4}), and Lemma \ref{L:sumRj}.
\end{proof}

\begin{Cor}\label{hitcor}
Let $A_0$ be the constant defined in Lemma \ref{L:sumRj}.  If there is a single particle at $x$ at time zero, where $0 < x < L(0) - 1$, then for $t \geq A_0$,
$$\P(R_{0,t} > 0) \asymp e^{\sqrt{2} x} \sin \bigg( \frac{\pi x}{L(0)} \bigg) t^{1/3} \exp(-(3\pi^2t)^{1/3}).$$
Likewise, if $0 < \alpha < \beta < 1$, then there are positive constants $C'_{\alpha, \beta}$ and $C''_{\alpha, \beta}$, depending on $\alpha$ and $\beta$ such that for all $t \geq A_0$,
\begin{align}
&C'_{\alpha, \beta} e^{\sqrt{2} x} \sin \bigg( \frac{\pi x}{L(0)} \bigg) t^{1/3} \exp(-(3\pi^2t)^{1/3}) \nonumber \\
&\qquad \leq \P(R_{\alpha t, \beta t} > 0) \leq C''_{\alpha, \beta} e^{\sqrt{2} x} \sin \bigg( \frac{\pi x}{L(0)} \bigg) t^{1/3} \exp(-(3\pi^2t)^{1/3}). \nonumber
\end{align}
\end{Cor}

\begin{proof}
Note that $j(0,x) \leq C h(0,x)$ when $x < L(0) - 1$ by (\ref{585}).  Therefore, by Lemma \ref{L:sumRj} with $s = 0$ and Markov's Inequality, $$\P(R_{\alpha t, \beta t} > 0) \leq \P(R_{0,t} > 0) \leq \E[R_{0,t}] \leq C(h(0,x) + j(0,x)) \leq C h(0,x).$$   For the lower bound, we use a standard second moment argument and apply Lemmas \ref{L:sumRj} and \ref{P:R2} to get
$$\P(R_{0,t} > 0) \geq \P(R_{\alpha t, \beta t} > 0) \geq \frac{(\E_{0,x}[R_{\alpha t, \beta t}])^2}{\E_{0,x}[R_{\alpha t, \beta t}^2]} \geq \frac{C h(0,x)^2}{h(0,x)} = C h(0,x).$$  The result follows.
\end{proof}

\section{Proofs of main results}\label{yagsec}

In this section, we prove Theorem \ref{yaglom} and Theorem \ref{extinct}.  The key to these proofs is Proposition \ref{survivefromL} below.  We first recall the following result due to Neveu \cite{nev87}.

\begin{Lemma}\label{nevlem}
Consider branching Brownian motion with drift $-\sqrt{2}$ and no absorption, started with a single particle at the origin.  For each $y \geq 0$, let $K(y)$ be the number of particles that reach $-y$ in a modified process in which particles are killed upon reaching $-y$.  Then there exists a random variable $W$, with $\P(0 < W < \infty) = 1$ and $\E[W] = \infty$, such that $$\lim_{y \rightarrow \infty} y e^{-\sqrt{2} y} K(y) = W \hspace{.1in}a.s.$$
\end{Lemma}

To prove Proposition \ref{survivefromL}, we will use the following result about the survival probability of a Galton-Watson process, which is Lemma 13 of \cite{bbm2}.

\begin{Lemma}\label{GW}
Let $(p_k)_{k=0}^{\infty}$ be a sequence of nonnegative numbers that sum to 1, and let $X$ be a random variable such that $\P(X = k) = p_k$ for all nonnegative integers $k$.  Let $q$ be the extinction probability of a Galton-Watson process with offspring distribution $(p_k)_{k=1}^{\infty}$ started with a single individual.  Then $$1 - q \geq \frac{2(\E[X] - 1)}{\E[X(X-1)]}.$$
\end{Lemma}

\begin{figure}\centering
\includegraphics[scale=1]{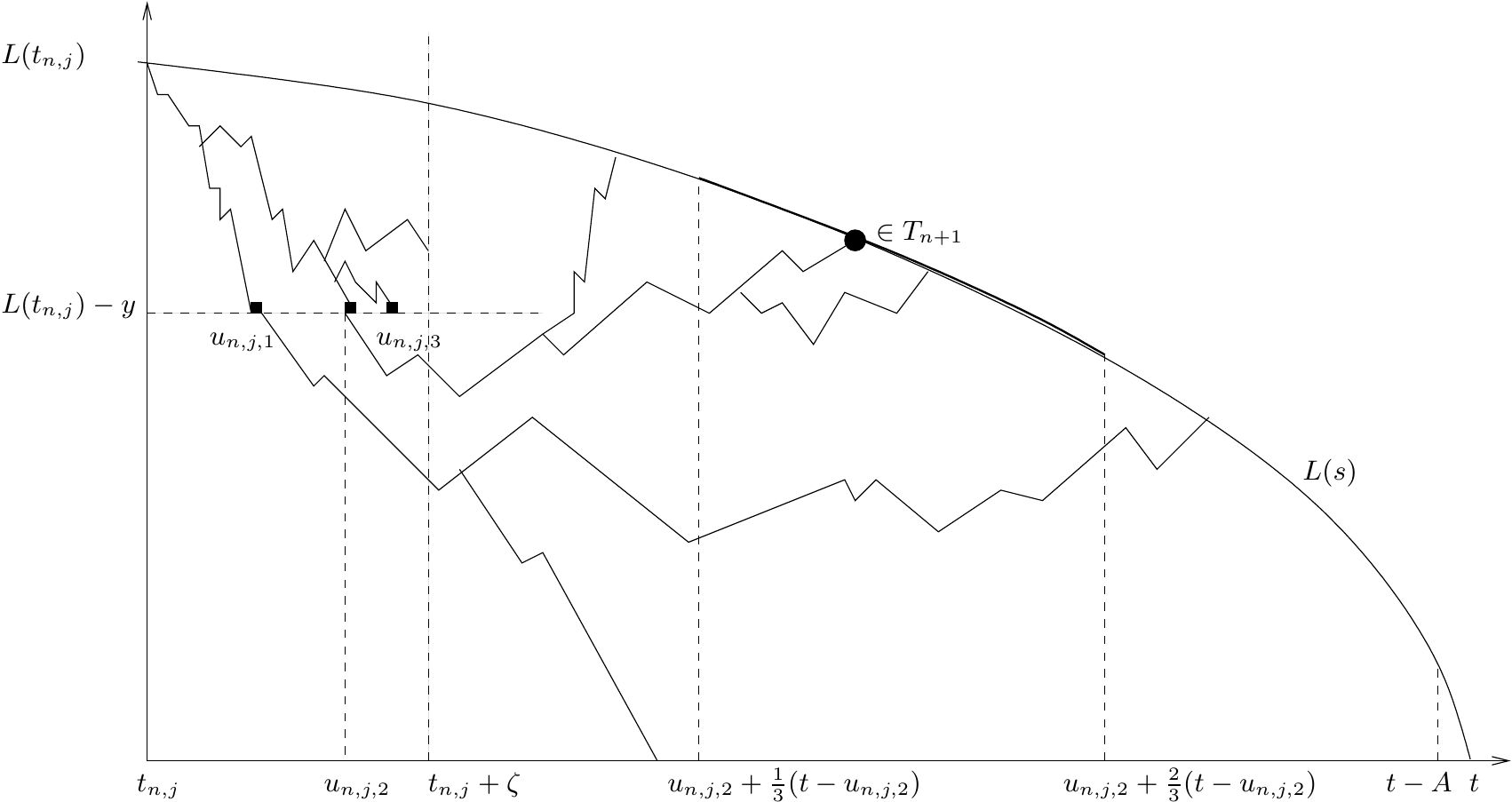}
\caption{ Construction of the branching process $T_n, n\ge 1$. Here we look at particle $j$ in generation $n$ alive at time $t_{n,j}$. It has three descendants that hit level $L(t_{n,j})-y$ figured by three squares. The second particle has a descendant that hits $L$ between time $u_{n,j,2}+(t-u_{n,j,2})/3$ and $u_{n,j,2}+2(t-u_{n,j,2})/3)$. The first of these descendants, indicated by a black dot, belongs to $T_{n+1}$. }
\end{figure}

\begin{Prop}\label{survivefromL}
Fix $t > 0$, and suppose that initially there is a single particle at $x = ct^{1/3}$.  Then there are constants $A > 0$ and $C > 0$ such that if $t \geq A$, the probability that there is at least one particle remaining at time $t$ is at least $C$.
\end{Prop}

\begin{proof}
We prove this result by constructing a branching process that resembles a discrete-time Galton-Watson process but allows individuals to have different offspring distributions.  We will show that the probability that this branching process survives is bounded below by a positive constant, and that survival of this branching process implies that the branching Brownian motion survives until at least time $t - A$.  This will in turn give the branching Brownian motion a positive probability of surviving until time $t$, which will imply the result.

Let $C' = C'_{1/3, 2/3}$, where $C'_{1/3, 2/3}$ is the constant from Corollary \ref{hitcor} with $\alpha = 1/3$ and $\beta = 2/3$.  Consider the setting of Lemma \ref{nevlem}, in which we have branching Brownian motion with drift $-\sqrt{2}$ and no absorption.  For $y > 0$, let $K(y)$ denote the number of particles that reach $-y$, if particles are killed upon reaching $-y$.  For $\zeta > 0$, let $K_{\zeta}(y)$ be the number of these particles that reach $y$ before time $\zeta$.  Because the random variable $W$ in Lemma \ref{nevlem} has infinite expected value, it follows from Lemma \ref{nevlem} and Fatou's Lemma that we can choose $y > 0$ sufficiently large that $$\E[y e^{-\sqrt{2} y} K(y)] \geq \frac{3 \cdot 2^{1/3} c}{C'}.$$  We can then choose a real number $\zeta > 0$ and a positive integer $M$ sufficiently large that
\begin{equation}\label{kzeta}
\E[y e^{-\sqrt{2} y} (K_{\zeta}(y) \wedge M)] \geq \frac{2 \cdot 2^{1/3} c}{C'}.
\end{equation}
Let $A_0$ be defined as in Corollary \ref{hitcor}.  Choose $A$ to be large enough that the following hold:
\begin{align}
&A \geq \max\{A_0 + \zeta, 2 \zeta\} \label{51} \\
&cA^{1/3} \geq 2y \label{52} \\
&cA^{1/3} - c(A - \zeta)^{1/3} \leq \frac{y}{2}. \label{53}
\end{align}
Let $t \geq A$, and let $L(s) = c(t-s)^{1/3}$ for $0 \leq s \leq t$.

We now construct the branching process inductively.  Let $T_0 = \{0\}$.  Suppose that $T_n = \{t_{n,1}, t_{n,2}, \dots, t_{n,m_n}\}$, which will imply that at the $n$th stage of the process, there are particles at positions $L(t_{n,1}), \dots, L(t_{n, m_n})$ at times $t_{n,1}, \dots, t_{n,m_n}$.  For $j = 1, 2, \dots, m_n$, if $t_{n,j} \geq t - A$, then we put $t_{n,j}$ in the set $T_{n+1}$.  If $t_{n,j} < t - A$, then we follow the trajectories after time $t_{n,j}$ of the descendants of the particle that reached $L(t_{n,j})$ at time $t_{n,j}$ until either time $t_{n,j} + \zeta$, or until the descendant particles reach $L(t_{n,j}) - y$, which is positive by (\ref{52}).  Denote the times, before time $t_{n,j} + \zeta$, at which descendant particles reach $L(t_{n,j}) - y$ by $u_{n,j,1} < \dots < u_{n,j,\ell_{n,j}}$.  For $\ell = 1, \dots, \ell_{n,j} \wedge M$, if at least one descendant of the particle that reaches $L(t_{n,j}) - y$ at time $u_{n,j,\ell}$ later reaches $L(s)$ at some time $s \in [u_{n,j,\ell} + (t - u_{n,j,\ell})/3, u_{n,j,\ell} + 2(t - u_{n,j,\ell})/3]$, then we put the smallest time $s$ at which this occurs in the set $T_{n+1}$.  For $n \geq 0$, let $Z_n$ be the cardinality of $T_n$.

The next step is to obtain bounds on the moments of $Z_1$ which are valid for all $t \geq A$.  Write $u_i = u_{0,1,i}$.  Then particles reach $L(0) - y$ at times $u_1, \dots, u_{\ell_{0,1}}$.  Observe that
\begin{equation}\label{Z1sum}
Z_1 = \xi_1 + \dots + \xi_{\ell_{0,1} \wedge M},
\end{equation}
where $\xi_i = 1$ if the particle that reaches $L(0) - y$ at time $u_i$ has a descendant that reaches $L(s)$ at some time $s \in [u_i + (t - u_i)/3, u_i + 2(t - u_i)/3]$ and $\xi_i = 0$ otherwise.  Let ${\cal G}$ be the $\sigma$-field generated by $u_1, \dots, u_{\ell_{0,1}}$.  By Corollary \ref{hitcor}, if $t \geq A$, then
\begin{align}\label{Xiui}
&C'e^{\sqrt{2}(x - y)} \sin \bigg( \frac{\pi(x-y)}{L(u_i)} \bigg) (t - u_i)^{1/3} \exp(-(3 \pi^2(t-u_i))^{1/3}) \nonumber \\
&\qquad \leq \P(\xi_i = 1|{\cal G}) \leq C e^{\sqrt{2}(x - y)} \sin \bigg( \frac{\pi(x-y)}{L(u_i)} \bigg) (t - u_i)^{1/3} \exp(-(3 \pi^2(t-u_i))^{1/3}).
\end{align}
Because $A \geq A_0 + \zeta$ by (\ref{51}), there is a constant $C$ such that if $t \geq A$ then
\begin{equation}\label{s1}
1 = e^{\sqrt{2} x} \exp(-(3 \pi^2 t)^{1/3}) \leq e^{\sqrt{2} x} \exp(-(3 \pi^2(t - u_i))^{1/3}) \leq e^{\sqrt{2} x} \exp(-(3 \pi^2(t - \zeta))^{1/3}) \leq C.
\end{equation}
Because $A \geq 2 \zeta$ by (\ref{51}), if $t \geq A$ then
\begin{equation}\label{s2}
(t/2)^{1/3} \leq (t - \zeta)^{1/3} \leq (t - u_i)^{1/3} \leq t^{1/3}.
\end{equation}
Therefore, using again that $A \geq 2 \zeta$, we get, when $t \geq A$,
\begin{equation}\label{s3}
\sin \bigg( \frac{\pi(x-y)}{L(u_i)} \bigg) = \sin \bigg( \frac{\pi(L(u_i) - x + y)}{L(u_i)} \bigg) \leq \frac{\pi(L(u_i) - x + y)}{L(u_i)} \leq \frac{\pi y}{L(u_i)} \leq \frac{\pi y}{c(t - \zeta)^{1/3}} \leq \frac{2^{1/3} \pi y}{c t^{1/3}}.
\end{equation}
By (\ref{53}), $$x - L(u_i) \leq L(0) - L(\zeta) = ct^{1/3} - c(t - \zeta)^{1/3} \leq y/2$$ for $t \geq A$.  Using this result and the fact that $\sin(x) \geq 2x/\pi$ for $0 \leq x \leq \pi/2$, we get
\begin{equation}\label{s4}
\sin \bigg( \frac{\pi(x-y)}{L(u_i)} \bigg) = \sin \bigg( \frac{\pi(L(u_i) - x + y)}{L(u_i)} \bigg) \geq \frac{2(L(u_i) - x + y)}{L(u_i)} \geq \frac{y}{L(u_i)} \geq \frac{y}{ct^{1/3}}.
\end{equation}
Combining (\ref{Xiui}), (\ref{s1}), (\ref{s2}), (\ref{s3}), and (\ref{s4}), we get
\begin{equation}\label{Pxi}
\frac{C'}{2^{1/3} c} ye^{-\sqrt{2} y} \leq \P(\xi_i = 1|{\cal G}) \leq C y e^{-\sqrt{2} y}.
\end{equation}
Because $\ell_{0,1}$ has the same distribution as $K_{\zeta}(y)$, it follows from (\ref{kzeta}), (\ref{Z1sum}), and (\ref{Pxi}) that
\begin{equation}\label{EZ1}
\E[Z_1] \geq \frac{C'}{2^{1/3} c} y e^{-\sqrt{2} y} \E[K_{\zeta}(y) \wedge M] \geq 2.
\end{equation}
From (\ref{Z1sum}), we see that $Z_1 \leq M$ so
\begin{equation}\label{EZ2}
\E[Z_1^2] \leq M^2 \leq C.
\end{equation}

For $n \geq 0$, let $q_{t,n} = \P(T_n = \emptyset)$.  Let $q_t = \lim_{n \rightarrow \infty} q_{t,n} = \P(T_n = \emptyset \mbox{ for some }n)$.  Let $p_t(k) = \P(Z_1 = k)$.  For $z \in [0,1]$, let $$\varphi_t(z) = \sum_{k=0}^{\infty} p_t(k)z^k.$$  Let $q_{t,*} = \min\{q \in [0,1]: \varphi_t(q) = q\}$, which is the probability that a Galton-Watson branching process goes extinct if each individual independently has $k$ offspring with probability $p_t(k)$.

Let $$q_* = \sup_{t > 0} q_{t,*}.$$  We claim that for all $t > 0$ and all $n \geq 0$, we have $q_{t,n} \leq q_*$.  We prove this claim by induction on $n$.  Because $q_{t,0} = 0$ for all $t > 0$, the claim is clear when $n = 0$.  Suppose the claim holds for some $n$.  Then by the induction hypothesis, $$\P(T_{n+1} = \emptyset|T_1 = \{s_1, \dots, s_k\}) = \prod_{j=1}^k  q_{t-s_j, n} \leq q_*^k.$$  Taking expectations of both sides gives $$q_{t,n+1} \leq \sum_{k=0}^{\infty} p_t(k) q_*^k = \varphi_t(q_*).$$  Because $\varphi_t(q_{t,*}) = q_{t,*}$ and $\varphi_t(1) = 1$, that fact that $z \mapsto \varphi_t(z)$ is nondecreasing and convex implies that if $z \geq q_{t,*}$, then $\varphi_t(z) \leq z$.  Therefore, since $q_* \geq q_{t,*}$, we have $\varphi_t(q_*) \leq q_*$.  Thus, $q_{t,n+1} \leq q_*$, and the claim follows by induction.

The claim implies that $q_t \leq q_*$ for all $t > 0$.  If $0 < t \leq A$, then $p_t(1) = 1$ and thus $q_{t,*} = 0$.  If $t \geq A$, then by Lemma \ref{GW} and equations (\ref{EZ1}) and (\ref{EZ2}),
$$1 - q_{t,*} \geq \frac{2(\E[Z_1] - 1)}{\E[Z_1(Z_1 - 1)]} \geq \frac{2(\E[Z_1] - 1)}{\E[Z_1^2]} \geq C.$$  It follows that $1 - q_* \geq C$, and therefore $1 - q_t \geq C$ for all $t \geq A$.

Thus, there is a constant $C$ such that, for all $t \geq A$, the probability that $T_n \neq \emptyset$ for all $n$ is at least $C$.  However, if $T_n \neq \emptyset$ for all $n$, then eventually some particle must reach $L(s)$ for some $s \in [t-A, t-A/3]$.  The probability that a particle reaching $L(s)$ for some $s \in [t-A, t-A/3]$ survives until time $t$ is bounded below by a constant.  The result follows.
\end{proof}

\begin{proof}[Proof of Theorem \ref{extinct}]
We first obtain an upper bound for the extinction time.  Let $\beta > 0$, and let $t_+ = t + \beta x^2$
where $t =\tau x^3$.  For $0 \leq s \leq t_+$, let $L_+(s) = c(t_+ - s)^{1/3}$.  Consider the process in which particles are killed at time $s$ if they reach $L_+(s)$.  The probability that the original process survives until time $t_+$ is bounded above by the probability that a particle is killed at $L_+(s)$ for some $s \in [0,t_+]$.
Note that $L_+(0) - x = c(t_+^{1/3} - t^{1/3}) \asymp \beta x^2 t^{-2/3} \asymp \beta$. Therefore, as soon as $x$ is large enough so that $t\ge A_0$ we can apply Corollary \ref{hitcor} to bound the probability that the original process survives until time $t_+$ by
\begin{equation}\label{starbd}
C e^{\sqrt{2} x} \sin \bigg( \frac{\pi x}{L_+(0)} \bigg) t_+^{1/3} e^{-(3 \pi^2 t_+)^{1/3}}.
\end{equation}
Observe that furthermore $$\sin \bigg(\frac{\pi x}{L_+(0)} \bigg) \leq \frac{
\pi (L_+(0) - x)}{L_+(0)} \leq C \beta t_+^{-1/3}.$$ and $$\exp \big( \sqrt{2} x - (3 \pi^2 t_+)^{1/3} \big) = \exp \big( -(3 \pi^2)^{1/3}(t_+^{1/3} - t^{1/3}) \big) \leq e^{-C' \beta},$$
for some positive constant $C'$.
Therefore, the probability in (\ref{starbd}) is at most $C \beta e^{-C' \beta}$, which is less than $\varepsilon/2$ for sufficiently large $\beta$.  For such $\beta$, we have
\begin{equation}\label{zetaup}
\P(\zeta < t_+) \geq 1 - \frac{\varepsilon}{2}
\end{equation}
for sufficiently large $x$.

To obtain the lower bound on the extinction time, let $t_- = t - \beta x^2$.
For $0 \leq s \leq t_-$, let $L_-(s) = c(t_- - s)^{1/3}$.  For $y > 0$ and $\zeta > 0$, let $K_{\zeta}(y)$ denote the number of particles that would be killed, if particles were killed upon reaching $x-y$ before time $\zeta$.  By Lemma \ref{nevlem}, we can choose $y$ and $\zeta$ sufficiently large and $\gamma > 0$ sufficiently small that $y \geq 2c^3 \beta + 1$ and
\begin{equation}\label{Kbig}
\P(K_{\zeta}(y) > \gamma y^{-1} e^{\sqrt{2} y}) > 1 - \frac{\varepsilon}{4}.
\end{equation}
Observe that for sufficiently large $x$,
\begin{equation}\label{ch1}
t_- - \zeta = t - \beta x^2 - \zeta \geq \frac{t}{2},
\end{equation}
which means for all $u \in (0, \zeta)$, $$x - L_-(u) = c[t^{1/3} - (t - \beta x^2 - u)^{1/3}] \leq \frac{c}{3} \bigg( \frac{t}{2} \bigg)^{-2/3} (\beta x^2 + \zeta) \leq c^3 \beta$$ for sufficiently large $x$.  Because $y \geq c^3 \beta + 1$, it follows that
\begin{equation}\label{hyp17}
x - y \leq L_-(u) - 1
\end{equation}
for all $u \in (0, \zeta)$, if $x$ is sufficiently large.

Now suppose a particle reaches $x - y$ at time $u \in (0, \zeta)$.  In view of (\ref{hyp17}), we can apply Corollary \ref{hitcor} to see that the probability that a descendant of this particle reaches $L(s)$ for some $s \in [u, u + (t_- - u)/2]$ is at least
\begin{equation}\label{survprob}
C e^{\sqrt{2}(x - y)} \sin \bigg( \frac{\pi(x-y)}{L_-(u)} \bigg)(t_- - u)^{1/3} \exp(-(3 \pi^2(t_- - u))^{1/3}).
\end{equation}
Using that $y \geq 2c^3\beta$ and that $\sin(x) \geq 2x/\pi$ for $0 \leq x \leq \pi/2$,
\begin{equation}\label{ch2}
\sin \bigg( \frac{\pi(x-y)}{L_-(u)} \bigg) = \sin \bigg( \frac{\pi(L_-(u) - x + y)}{L_-(u)} \bigg) \geq \frac{2(L_-(u) - x + y)}{L_-(u)} \geq \frac{2(y - c^3 \beta)}{c t^{1/3}} \geq \frac{y}{ct^{1/3}}.
\end{equation}
Also, for sufficiently large $x$, we have $t^{1/3} - (t - \beta x^2 - u)^{1/3} \geq (1/3)t^{-2/3} \cdot \beta x^2 = (c^2/3) \beta$, and so
\begin{align}\label{ch3}
\exp(-(3 \pi^2(t_- - u))^{1/3}) &= \exp(-(3 \pi^2 t)^{1/3}) \exp((3 \pi^2)^{1/3}[t^{1/3} - (t - \beta x^2 - u)^{1/3}]) \nonumber \\
&\geq \exp(-(3 \pi^2 t)^{1/3}) \exp((3 \pi^2)^{1/3} c^2 \beta/3).
\end{align}
Recall also that
\begin{equation}\label{ch4}
e^{\sqrt{2} x} e^{-(3 \pi^2 t)^{1/3}} = 1.
\end{equation}
By (\ref{ch1}), (\ref{ch2}), (\ref{ch3}), and (\ref{ch4}), for sufficiently large $x$, the probability in (\ref{survprob}) is at least
\begin{equation}\label{survprob2}
C y e^{-\sqrt{2} y} e^{(3 \pi^2)^{1/3} c^2 \beta/3},
\end{equation}
where the constant $C$ does not depend on $\beta$.  By Proposition \ref{survivefromL}, the probability that a descendant of this particle survives until time $t_-$ is also bounded below by (\ref{survprob2}), with a different positive constant $C$.
Therefore, conditional on the event that $K_{\zeta}(y) > \gamma y^{-1} e^{\sqrt{2} y}$, the probability that some particle survives until $t_-$ is at least
$$1 - (1 - Cye^{-\sqrt{2} y} e^{(3 \pi^2)^{1/3}c^2 \beta/3})^{\gamma y^{-1} e^{\sqrt{2} y}}.$$ Using the inequality $1-a \leq e^{-a}$ for $a \in \R$, we see that this expression is bounded below by $$1 - \exp \big( -C \gamma e^{(3 \pi^2)^{1/3}c^2 \beta/3} \big)$$ and therefore is at least $1 - \varepsilon/4$ if $\beta$ is chosen to be large enough.  Combining this result with (\ref{Kbig}) gives that for such $\beta$,
\begin{equation}\label{zetalow}
\P(\zeta > t_-) \geq 1 - \frac{\varepsilon}{2}
\end{equation}
for sufficiently large $x$.  The result follows from (\ref{zetaup}) and (\ref{zetalow}).
\end{proof}

\begin{proof}[Proof of Theorem \ref{yaglom}]
First, suppose that $t \geq \max\{A_0, 2A\}$, where $A_0$ is the constant from Corollary \ref{hitcor} and $A$ is the constant from Proposition \ref{survivefromL}.  Suppose also that $0 < x < ct^{1/3} - 1$.  For $0 \leq s \leq t$, let $L(s) = c(t-s)^{1/3}$.  Consider a modification of the branching Brownian motion in which particles, in addition to getting killed at the origin, are killed if they reach $L(s)$ for some $s \in [0,t]$.  Let $R_1$ be the number of particles that are killed at $L(s)$ for some $s \in (0,t)$, and let $R_2$ be the number of particles that are killed at $L(s)$ for some $s \in (0, t/2)$.  By Corollary \ref{hitcor},
\begin{equation}\label{yagup}
\P(R_1 > 0) \leq C e^{\sqrt{2} x} \sin \bigg( \frac{\pi x}{L(0)} \bigg)t^{1/3} e^{-(3 \pi^2 t)^{1/3}}.
\end{equation}
In this modified process, all particles disappear before time $t$.  Therefore, the only way to have $\zeta > t$ is to have, in the modified process, a particle killed at $L(s)$ for some $s \in (0, t)$.   The upper bound in (\ref{E:ext proba 1}) thus follows from the upper bound in (\ref{yagup}).

Likewise, Corollary \ref{hitcor} implies that
$$\P(R_2 > 0) \geq C e^{\sqrt{2} x} \sin \bigg( \frac{\pi x}{L(0)} \bigg)t^{1/3} e^{-(3 \pi^2 t)^{1/3}}.$$  By Proposition \ref{survivefromL}, a particle that reaches $L(s)$ at time $s \in (0, t/2)$ has a descendant alive at time $t$ with probability greater than $C$.  This implies the lower bound in (\ref{E:ext proba 1}).

Next, suppose $0 < t < \max\{A_0, 2A\}$ and $0 < x < ct^{1/3} - 1$.  Let $(B(s), s \geq 0)$ be standard Brownian motion with $B(0) = x$.  The probability that the branching Brownian motion survives until time $t$ is bounded below by $P(B(s) > 0 \mbox{ for all }s \in [0,t])$ and is bounded above by $e^t P(B(s) > 0 \mbox{ for all }s \in [0,t])$.  Because both $x$ and $t$ are bounded above by a positive constant, both of these expressions are of the order $x$, as are the expressions on the left-hand side and the right-hand side of (\ref{E:ext proba 1}).  Consequently, (\ref{E:ext proba 1}) holds when $0 < t < \max\{A_0, 2A\}$.

Finally, (\ref{E:ext proba 2}) follows from (\ref{E:ext proba 1}) by fixing $x > 0$ and letting $t \rightarrow \infty$.
\end{proof}

\bigskip \noindent
Julien Berestycki: \\
Universit\'e Pierre et Marie Curie.  LPMA / UMR 7599, Bo\^ite courrier 188. 75252 Paris Cedex 05

\medskip \noindent Nathana\"el Berestycki: \\
DPMMS, University of Cambridge. Wilberforce Rd., Cambridge CB3 0WB 

\medskip \noindent
Jason Schweinsberg:\\
University of California at San Diego, Department of Mathematics. 9500 Gilman Drive; La Jolla, CA 92093-0112

\end{document}